\documentclass[11pt]{amsart}
\usepackage{amsfonts,latexsym,amsthm,amssymb,graphicx}
\usepackage[all]{xy}
\usepackage{hyperref}
\usepackage[usenames]{color} 
\usepackage{manfnt}

\DeclareFontFamily{OT1}{rsfs}{}
\DeclareFontShape{OT1}{rsfs}{n}{it}{<-> rsfs10}{}
\DeclareMathAlphabet{\mathscr}{OT1}{rsfs}{n}{it}
\DeclareRobustCommand{\mymarginpar}[1]{%
 \marginpar[\raggedleft#1]{\raggedright#1}}

\CompileMatrices

\newtheorem{theorem}{Theorem}[section]
\newtheorem{lemma}[theorem]{Lemma}
\newtheorem{corol}[theorem]{Corollary}

{\theoremstyle{definition} }
{\theoremstyle{remark} \newtheorem{remark}[theorem]{Remark}
\newtheorem{example}[theorem]{Example}}

\newcommand{\Rbb}{{\mathbb{R}}}
\newcommand{\Zbb}{{\mathbb{Z}}}
\newcommand{\cL}{{\mathscr L}}
\newcommand{\cO}{{\mathscr O}}
\newcommand{\cT}{{\mathscr T}}
\newcommand{\cU}{{\mathscr U}}
\newcommand{\hU}{{\widehat U}}
\newcommand{\hcT}{\widehat{\mathscr T}}
\newcommand{\hcU}{\widehat{\mathscr U}}
\newcommand{\hsigma}{\hat\sigma}
\newcommand{\hM}{{\widehat M}}
\newcommand{\hN}{{\widehat N}}
\newcommand{\hp}{{\hat p}}
\newcommand{\ua}{\underline a}
\newcommand{\ue}{\underline e}
\newcommand{\uv}{\underline v}
\newcommand{\uX}{\underline X}
\newcommand{\huv}{\hat{\underline v}}
\newcommand{\Til}[1]{{\widetilde{#1}}}
\newcommand{\qede}{\hfill$\lrcorner$}
\DeclareMathOperator{\Vol}{Vol}
\DeclareMathOperator{\hVol}{\widehat \Vol}

\title{
Segre classes as integrals over polytopes
}
\author{Paolo Aluffi}
\address{
Mathematics Department, 
Florida State University,
Tallahassee FL 32306, U.S.A.
}
\email{aluffi@math.fsu.edu}

\begin{document}

\begin{abstract}
We express the Segre class of a monomial scheme---or, more generally, a scheme
monomially supported on a set of divisors cutting out complete intersections---in 
terms of an integral computed over an associated body in euclidean space.
The formula is in the spirit of the classical Bernstein-Kouchnirenko theorem computing 
intersection numbers of equivariant divisors in a torus in terms of mixed volumes, but 
deals with the more refined intersection-theoretic invariants given by Segre classes, 
and holds in the less restrictive context of `r.c.~monomial schemes'.
\end{abstract}

\maketitle

\section{Introduction}\label{intro}

\subsection{}
Let $V$ be a variety (or more generally a pure-dimensional scheme of finite type over an 
algebraically closed field) and let $X_1,\dots, X_n$ be effective Cartier
divisors in $V$. A choice of nonnegative multiplicities $(a_1,\dots, a_n)\in \Zbb^n$
determines an effective divisor, obtained by taking $X_i$ with multiplicity~$a_i$. We call such a
divisor a {\em monomial\/} in the $X_i$'s, and we call {\em monomial scheme\/} (w.r.t.~the
chosen divisors) an intersection of such monomials. In the literature (e.g., \cite{MR2165388}) 
this terminology is reserved for the case in which $V$ is nonsingular and the divisors $X_i$ are 
nonsingular and meet with normal crossings. We will refer to this as the `standard' situation.
In this paper we consider a substantially more general case, in which $V$ may be any
scheme as above and the divisors $X_1,\dots,X_n$ are only required to intersect `completely',
in the sense that their local defining equations form regular sequences. The resulting
monomial scheme is a `r.c.~monomial scheme' in the terminology of~\cite{arXiv:1310.1261}.
(We should stress that a r.c.~monomial scheme is {\em not\/} necessarily a complete 
intersection!)

A monomial scheme $S$ is determined by an $n$-uple of divisors and the 
choice of finitely many lattice points in $\Zbb^n$. We call the complement of the convex 
hull of the positive orthants translated at these points the `Newton region' corresponding 
to (this description of) $S$. In \cite{Scms}, Conjecture~1, we proposed a formula for the 
{\em Segre class\/} of a monomial subscheme~$S$ in the standard situation, and 
proved this formula in several representative cases. The purpose of this note is to prove 
a generalization of the full statement of Conjecture~1 in~\cite{Scms}.

\begin{theorem}\label{main}
Let $S\subseteq V$ be a r.c.~monomial scheme with respect to a choice of $n$ divisors
$X_1,\dots, X_n$, and let $N$ be the corresponding Newton region. Then
\begin{equation}\label{eq:main}
s(S,V)=\int_N \frac{n! X_1\cdots X_n\, da_1\cdots da_n}
{(1+a_1 X_1+\cdots +a_n X_n)^{n+1}}\quad.
\end{equation}
\end{theorem}

Note that different monomial schemes may have the same Newton region. According 
to~\eqref{eq:main}, such schemes have the same Segre class. This phenomenon is
due to the fact that the corresponding ideals have the same integral closure, cf.~Remark~2.5
in~\cite{Scms}.

\subsection{}
The right-hand side of \eqref{eq:main} is interpreted by evaluating the integral formally 
with parameters $X_1,\dots$, $X_n$; the result is a rational function in $X_1,\dots,X_n$, 
with a well-defined
expansion as a power series in these variables, all of whose terms may be interpreted
as intersection products of the corresponding divisors in $V$. These products are 
naturally supported on subschemes of $S$ (cf.~Lemma~\ref{supp}). 
The statement is that evaluating the terms 
of the series as these intersection products gives the Segre class $s(S,V)$ of $S$ in $V$
as an element of the Chow group of $S$. 

For a thorough treatment of Segre classes and their role in Fulton-MacPherson intersection 
theory, see~\cite{85k:14004}. Segre classes are one of the basic ingredients in the definition
of the intersection product (\cite{85k:14004}, Proposition 6.1 (a)). Residual intersection 
formulas are naturally written in terms of Segre classes (\cite{85k:14004}, Chapter 9). 
Many problems in enumerative geometry can be phrased in terms of computations of 
Segre classes. Interesting invariants of singularities, such as Milnor classes or 
Donaldson-Thomas invariants, may also be expressed in terms of Segre classes.
While the schemes considered in Theorem~\ref{main} may be too special for direct
applications of this type, the approach to the computation of Segre classes presented
in this result appears to be completely novel. We hope of course that it will be possible
to extend the scope of~\eqref{eq:main} to yet more general schemes.

We refer the reader to~\cite{Scms} for further contextual remarks 
concerning Segre classes and for examples illustrating~\eqref{eq:main} in the 
standard situation. In~\cite{Scms}, the formula is established
in the case $n=2$, and for several families of examples for arbitrary $n$. Also, the
formula is stated in~\cite{Scms} after push-forward to the ambient variety, and limited
to standard monomial schemes. In this paper the formula is proved as an equality of 
classes in the Chow group of $S$, for any~$n$, and for the larger class of r.c.~monomial 
schemes. This is a substantially more general setting: the ambient scheme $V$ and
the divisors $X_i$ need not be smooth, nor do the $X_i$ need to meet transversally.

\subsection{}
The proof of Theorem~\ref{main} is given in \S\ref{proof}. It is based on the birational
invariance of Segre classes and the fact that r.c.~monomial schemes may be 
principalized by a sequence of blow-ups at r.c.~monomial centers of codimension~$2$.
This is proven by C.~Harris~(\cite{arXiv:1310.1261}), extending the analogous result
for standard monomial schemes due to R.~Goward (\cite{MR2165388}).
This fact and an explicit computation in the principal
case reduce the proof to showing that the integral appearing in~\eqref{eq:main}
is preserved under blow-ups at codimension~$2$ (r.c.) monomial centers. This in turn 
follows from an analysis of triangulations associated with these blow-ups.

In~\cite{arXiv:1308.4152} we apply Theorem~\ref{main} to the computation of 
{\em multidegrees\/} of r.c.~monomial rational maps. This is one point of contact
of the result presented here with the existing literature: in the case
of (standard) monomial rational maps between projective spaces, the multidegrees
may be computed by {\em mixed volumes\/} of Minkowski sums of 
polytopes as an application of the classical Bernstein-Kouchnirenko theorem
(see {\em e.g.,\/} \cite{MR2221122}, \S4, or \cite{Dolgachev}, \S3.5).
In the application reviewed in~\cite{arXiv:1308.4152}, Theorem~\ref{main} leads
to an alternative expression for the multidegrees, which reproduces the volume
computation for the top multidegree, and generalizes it to the r.c.~setting and to 
rational maps on any projective variety. Formulas for the other multidegrees lead 
to integral expressions for the mixed volumes appearing in Bernstein-Kouchnirenko. 

It is natural to ask whether a result such as Theorem~\ref{main} may hold for 
non-monomial schemes. One possibility is that an integral formula analogous 
to~\eqref{main} may hold with $N$ replaced by a suitable body; this would be in
line with recent results for intersection numbers of divisors on open varieties,
as in~\cite{MR2950767}. (In fact, these results were our first motivation to look
for formulas for Segre classes in terms of `volumes', which led to the formulation
of Theorem~\ref{main}.)
For a different viewpoint, in~\cite{arXiv:1308.4153} we 
propose a formulation of the result 
(in the standard monomial setting) in terms of the {\em log canonical thresholds\/} 
of suitable extensions of the ideal defining the subscheme. This expression makes
sense for more general schemes, and may offer a candidate for a generalization
of Theorem~\ref{main}. 
\smallskip

{\em Acknowledgment.} The author's research is partially supported by a Simons 
collaboration grant. 

%%%

\section{Proof}\label{proof}

\subsection{}
We begin by recalling R.~Goward's theorem, and C.~Harris's generalization. 
First assume that we are in the standard
monomial situation: $V$ is a nonsingular variety, $X_1,\dots, X_n$
are nonsingular divisors meeting with normal crossings, and $S$ is the intersection of
monomial hypersurfaces, i.e., effective divisors $D_j$ supported on $\cup_i X_i$. 
According to Theorem~2 
in~\cite{MR2165388}, there exists a sequence of blow-ups at nonsingular centers
producing a proper birational morphism $\rho: V' \to V$ such that $\rho^{-1}(S)$ is
a divisor with normal crossings. Further, as explained in \S4 of \cite{MR2165388},
the centers of the blow-ups may all be chosen to be codimension-$2$ intersections of 
(proper transforms of) the original components $X_i$ and of the exceptional divisors 
produced in the process. In~\cite{arXiv:1310.1261}, C.~Harris shows that both statements
generalize verbatim to the {\em r.c.~monomial case.\/} Here no nonsingularity restrictions
are posed on~$V$ or on the divisors $X_1,\dots,X_n$. The divisors meet with~{\em 
regular crossings\/} if for all $A\subseteq \{1,\dots,n\}$ and all $p\in \cap_{i\in A} X_i$,
the local equations for $X_i$, $i\in A$, form a regular sequence at $p$ 
(\cite{arXiv:1310.1261}). A {\em r.c.~monomial scheme\/} is a scheme defined by
effective divisors supported on $\cup_i X_i$, where the $X_i$ meet with regular crossings.
According to Theorem~1 in~\cite{arXiv:1310.1261}, every r.c.~monomial scheme may
be principalized by a sequence of blow-ups at centers of codimension~2; as in 
Goward's result, these centers may be chosen to be intersections of proper transforms
of the $X_i$'s and of the exceptional divisors.

\begin{remark}\label{nota}
Let $X_1,\dots, X_n$ be a set of divisors meeting with regular crossings, and let 
$\pi: \Til V \to V$ be the blow-up along the intersection of two of these hypersurfaces; 
without loss of generality this is $X_1\cap X_2$, and (by definition of regular crossings)
it is a regularly embedded subscheme of codimension~$2$. Then 
\begin{itemize}
\item The proper transforms $\Til X_i$ of the components $X_i$, together with the exceptional 
divisor $E$, form a divisor with regular crossings (\cite{arXiv:1310.1261}, Proposition~3);
\item $\Til X_i=\pi^{-1}(X_i)$ for $i\ge 3$;
\item If $D$ is a monomial in the $X_i$'s, then $\pi^{-1}(D)$ is a monomial in the collection
$E$, $\Til X_i$;
\item In fact, if $D=\sum a_i X_i$, then $\pi^{-1}(D)=(a_1+a_2) E + \sum_i a_i \Til X_i$.
\end{itemize}
In view of the second point, we will write $X_i$ for $\Til X_i=\pi^{-1}(X_i)$ for $i\ge 3$;
this abuse of notation is good mnemonic help when using the projection formula. 
For example, the projection formula gives 
$\pi_*(\Til X_1 \cdot \Til X_3) = \pi_*(\Til X_1 \cdot  \pi^*(X_3)) =X_1\cdot X_3$.
We find this easier to parse if we write $\pi_*(\Til X_1 \cdot X_3) =X_1\cdot X_3$,
particularly as
$
\pi_*(\Til X_1\cdot \Til X_2)=0
$
since $\Til X_1\cap \Til X_2=\emptyset$ to begin with. Also note that
$
\pi_*(E\cdot \Til X_2)=X_1\cdot X_2$ and $\pi_*(E\cdot X_i)=0
$
for $i\ge 3$.
\qede\end{remark}

As remarked here, at each step in the sequence considered by Harris the inverse 
image of~$S$ is a monomial scheme with respect to the collection of proper transforms 
of the $X_i$'s and of the previous exceptional divisors, and the next blow-up is performed 
along the intersection of two of these hypersurfaces. In order to prove Theorem~\ref{main},
therefore, it suffices to prove the following two lemmas.

\begin{lemma}\label{prince}
Let $S$ be a r.c.~monomial scheme, and assume $S$ is a divisor.
Then \eqref{eq:main} holds for $S$.
\end{lemma}

\begin{lemma}\label{industep}
Let $S$ be a r.c.~monomial scheme, and let $\pi:\Til V \to V$
be the blow-up along $X_1\cap X_2$. Then if \eqref{eq:main} holds for $\pi^{-1}(S)$,
then it also holds for $S$.
\end{lemma}

The proofs of these two lemmas are given in the rest of this section, after some needed
preliminaries. As pointed out above, these two lemmas imply Theorem~\ref{main}.

\subsection{Integrals and triangulations}\label{intra}
The integral appearing in \eqref{eq:main} may be computed in terms of a triangulation
of $N$. An $n$-dimensional simplex in $\Rbb^n$ is the convex hull of a set of $n+1$ 
points (its vertices) not contained in a hyperplane. Points are denoted by underlined
letters: $\uv=(v_1,\dots, v_n)$. The notation $\uv\cdot \uX$ stands for 
$v_1 X_1+\cdots + v_n X_n$. We let $\ue_1=(1,0,\dots, 0)$, \dots, $\ue_n=(0,\dots, 0, 1)$.

We also allow for the possibility that some of the vertices are at infinity, and
we denote by $\ua_i$ the point at infinity in the direction of $\ue_i$. Thus, the
simplex $T$ with `finite' vertices $\uv_0,\dots, \uv_r$ and `infinite' vertices
$\ua_{i_1},\dots, \ua_{i_{n-r}}$ is defined by
\[
T = \left\{\sum_{j=0}^r \lambda_j \uv_j+ \sum_{k=1}^{n-r} \mu_k \ue_{i_k}\quad | \quad
\text{$\forall j,k: \lambda_j\ge 0, \mu_k\ge 0$, and $\sum_j \lambda_j=1$}\right\}\quad.
\]
Each simplex $T$ has a normalized volume $\hVol(T)$, defined as the normalized 
volume of the (finite) simplex obtained by projecting along its infinite directions. 
(The {\em normalized\/} volume is the ordinary Euclidean volume times the factorial
of the dimension.)

\begin{example}\label{ex:colum}
The simplex $T$ with vertices $\uv_0=(0,0,1), \uv_1=(1,0,2), \uv_2= (0,2,3)$ 
and $\underline a_3$ (at infinity)
\begin{center}
\includegraphics[scale=.5]{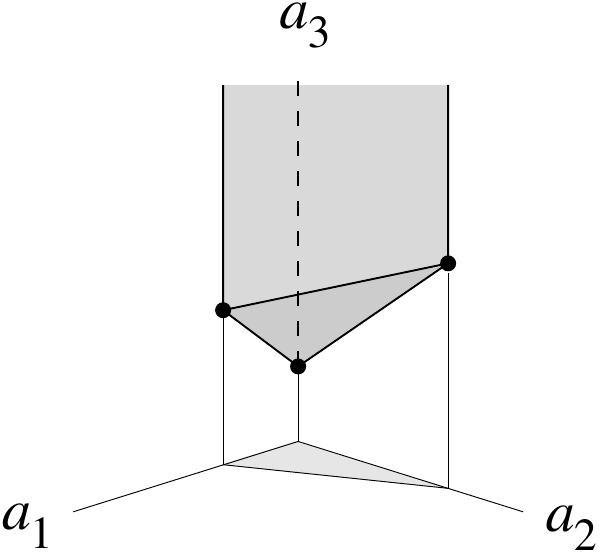}
\end{center}
has normalized volume $\hVol(T)=2$.
\qede\end{example}

We can associate with every simplex a contribution to the integral in~\eqref{eq:main}:

\begin{lemma}\label{simplexco}
If $T$ has finite vertices $\uv_0,\dots, \uv_r$ and infinite vertices
$\ua_{i_1},\dots, \ua_{i_{n-r}}$, then
\[
\int_T \frac{n! X_1\cdots X_n\, da_1\cdots da_n}
{(1+a_1 X_1+\cdots +a_n X_n)^{n+1}}
=\frac{\hVol(T)\, X_1\cdots X_n}
{\prod_{\ell=0}^r (1+\uv_\ell\cdot \uX) \prod_{j=1}^{n-r} X_{i_j}}\quad.
\]
\end{lemma}

\begin{proof}
This is Proposition~3.1 in~\cite{Scms}.
\end{proof}

Note that the numerator simplifies to give a multiple of the product of the parameters
$X_i$ corresponding to the `finite' part of the simplex.

\begin{example}
For the simplex in Example~\ref{ex:colum}, 
\begin{multline*}
\int_T \frac{3! X_1X_2 X_3\, da_1 da_2 da_3}
{(1+a_1 X_1+\cdots +a_n X_n)^4}
=\frac{\hVol(T)\, X_1X_2 X_3}
{(1+X_3)(1+X_1+2X_3)(1+2X_2+3X_3)\, X_3}\\
=\frac{2\, X_1X_2}
{(1+X_3)(1+X_1+2X_3)(1+2X_2+3X_3)}
\quad.
\end{multline*}
\end{example}

It immediately follows from Lemma~\ref{simplexco} that the integral over the whole positive
orthant $\Rbb^n_{\ge 0}$ equals~$1$. Also:

\begin{corol}\label{pows}
For every monomial scheme $S$, the integral
\[
\int_N \frac{n! X_1\cdots X_n\, da_1\cdots da_n}{(1+a_1 X_1+\cdots +a_n X_n)^{n+1}}
\]
is a rational function in $X_1,\dots,X_n$. It may be expanded as a power series in 
$X_1,\dots, X_n$ with integer coefficients.
\end{corol}

\begin{proof}
Triangulate $N$, then apply Lemma~\ref{simplexco}.
\end{proof}

\subsection{Residual intersection}
Let $S$ be a monomial scheme, realized as the intersection of monomials $D_1,\dots, D_r$.
Assume all $D_i$'s contain a fixed monomial divisor $D=\sum_i d_i X_i$; we obtain a
{\em residual\/} monomial scheme $R$ by intersecting the residuals to $D$ in $D_i$:
\[
R:= (D_1-D) \cap \cdots \cap (D_r-D)\quad.
\]
The {\em residual intersection\/} formula in intersection theory (cf.~\cite{85k:14004}, 
Proposition 9.2) gives a relation between the Segre classes of $S$, $D$, and $R$.
This formula should be expected to have a counterpart in terms of integrals.

\begin{lemma}
With $S$, $D$, $R$ as above:
\begin{itemize}
\item The Newton region $N_S$ for $S$ is the intersection of the positive orthant 
with the translate by $(d_1,\dots,d_n)$ of the Newton region $N_R$ for $R$:
\[
N_R=\{ (v_1,\dots, v_n)\in \Rbb^n_{\ge 0} \, | \, (v_1+d_1,\dots, v_r+d_r)\in N_S\}\quad.
\]
\item We have the equality
\begin{multline}\label{eq:compfor}
\int_{N_S} \frac{n! X_1\cdots X_n\, da_1\cdots da_n}{(1+a_1 X_1+\cdots +a_n X_n)^{n+1}}\\
=\frac{\sum_i d_i X_i}{(1+\sum_i d_i X_i)} + \frac 1{(1+\sum_i d_i X_i)} \left(
\int_{N_R} \frac{n! X_1\cdots X_n\, da_1\cdots da_n}{(1+a_1 X_1+\cdots +a_n X_n)^{n+1}}
\otimes \cO(\sum_i d_i X_i)\right)\quad.
\end{multline}
\end{itemize}
\end{lemma}

The notation introduced in~\cite{MR96d:14004}, \S2, is used in this statement (and will be
used in the following): for a line bundle $\cL$ and a class $A=\sum_i a^{(i)}$ in the Chow group, 
where $a^{(i)}$ has codimension~$i$ in the ambient scheme $V$, $A\otimes \cL$ denotes the 
class $\sum_i c(\cL)^{-i}\cap a^{(i)}$. This notation determines an action of Pic on the Chow
group, and is compatible with the effect of ordinary
tensors on Chern classes, cf.~Propositions~1 and~2 in~\cite{MR96d:14004}.

\begin{proof}
The first assertion is immediate.

For the second, note that the {\em complement\/} $N'_S$ of $N_S$ in the positive orthant 
is precisely the translate of the complement $N'_R$ by $(d_1,\dots,d_n)$. Since the
integral over the positive orthant is $1$, verifying the stated formula is equivalent to
verifying that
\begin{multline}\label{eq:compforc}
\int_{N'_S} \frac{n! X_1\cdots X_n\, da_1\cdots da_n}{(1+a_1 X_1+\cdots +a_n X_n)^{n+1}}\\
=\frac 1{(1+\sum_i d_i X_i)} \left(
\int_{N'_R} \frac{n! X_1\cdots X_n\, da_1\cdots da_n}{(1+a_1 X_1+\cdots +a_n X_n)^{n+1}}
\otimes \cO(\sum_i d_i X_i)\right)\quad.
\end{multline}
By the formal properties of the $\otimes$ operation (cf.~\cite{MR96d:14004}, Proposition~1)
\[
\frac{X_1\cdots X_n}{(1+a_1 X_1+\cdots +a_n X_n)^{n+1}}
\otimes \cO(\sum_i d_i X_i)
=\frac{(1+\sum_i d_i X_i)\, X_1\cdots X_n}{(1+(a_1+d_1) X_1+\cdots +(a_n+d_n) X_n)^{n+1}},
\]
showing that the right-hand side of~\eqref{eq:compforc} equals
\[
\int_{N'_R} \frac{n! X_1\cdots X_n\, da_1\cdots da_n}
{(1+(a_1+d_1) X_1+\cdots +(a_n+d_n) X_n)^{n+1}}\quad.
\]
It is now clear that this equals the left-hand side, since $N'_S$ is the translate of $N'_R$
by $(d_1,\dots, d_n)$.
\end{proof}

\begin{corol}\label{remdiv}
With notation as above, formula \eqref{eq:main} is true for $S$ in $A_*S$ if and only if it is 
true for $R$.
\end{corol}

\begin{proof}
Compare \eqref{eq:compfor} with the formula for residual intersections of Segre classes,
in the form given in ~\cite{MR96d:14004}, Proposition~3.
\end{proof}

\subsection{Proof of Lemma~\ref{prince}}
Let $D_1,\dots, D_r$ be monomials, and assume $S=D_1\cap \cdots\cap D_r$ is
a divisor. Note that this may happen even if $r> 1$: for example, suppose $X_1,
X_2, X_3$ are divisors meeting with normal crossings, and $X_1\cap X_2=\emptyset$.
If $D_1=X_1+X_3$ and $D_2=X_2+X_3$, then $S=D_1\cap D_2$ is a divisor, in fact, 
$X_3$. However, the Newton region of this representation of $S$ (depicted to the left)
includes an infinite column that is not present in the representation as $X_3$ (on the right).
\begin{center}
\includegraphics[scale=.5]{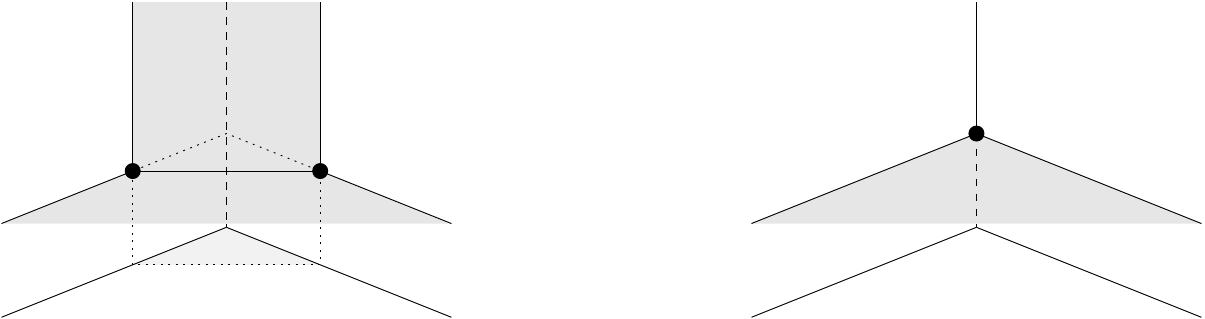}
\end{center}

We have to verify that if $D_1\cap \cdots \cap D_r=D$ is a monomial divisor, then
these two representations lead to the same integral. By Corollary~\ref{remdiv}, we 
may in fact eliminate the common factor $D$ in the monomials
$D_1,\dots, D_r$, and we are reduced to showing that if $D_1,\dots, D_r$ have 
{\em empty\/} intersection, then
\[
\int_N \frac{n! X_1\cdots X_n\, da_1\cdots da_n}
{(1+a_1 X_1+\cdots +a_n X_n)^{n+1}} = 0\quad.
\]
This follows immediately from the following more general statement.

\begin{lemma}\label{supp}
Let $D_1,\dots, D_r$ be monomials in $X_1,\dots, X_n$, and let $S=D_1\cap\cdots\cap D_r$,
with Newton region $N$. Then the class computed by the integral in~\eqref{eq:main},
\[
\int_N \frac{n! X_1\cdots X_n\, da_1\cdots da_n}
{(1+a_1 X_1+\cdots +a_n X_n)^{n+1}}
\]
is supported on $S$.
\end{lemma}

In particular, the class equals $0$ if $S=\emptyset$. We will in fact prove that the
class computed by the integral is a sum of classes obtained by applying Chow operators 
to classes of subschemes of $S$. As such, the integral defines an element in the Chow 
group $A_*S$.

\begin{proof}
The integral may be computed by triangulating $N$ and applying Lemma~\ref{simplexco}.
Thus, it suffices to show that if $T$ is a simplex contained in $N$ with finite vertices 
$\uv_0,\dots, \uv_r$ and infinite vertices $\ua_j$, $j\in J$, then the support of the class
\[
\frac{\hVol(T)\, X_1\cdots X_n}
{\prod_{\ell =0}^r (1+\uv_\ell\cdot \uX) \prod_{j\in J} X_j}
\]
is contained in $S$. Hence, it suffices to show that the intersection product
$\prod_{j\not\in J} X_j$ is supported on $S$, i.e., that $\cap_{j\not\in J} X_j$ is
contained in $S$. The coordinates $\ua_j$, $j\not \in J$, span the subspace
containing the (bounded) projection of $T$ along its unbounded directions.
Since $T\subseteq N$, the simplex spanned by $\ue_j$, $j\not\in J$, and
$\ua_j$, $j\in J$, is contained in $N$. This is the Newton region for  the
ideal generated by $X_j$, $j\not\in J$. It follows that this ideal contains the ideal
of $S$, concluding the proof.
\end{proof}

\subsection{Proof of Lemma~\ref{industep}}\label{prindu}
With notation as in the beginning of this section (see Remark~\ref{nota}), we have to
prove that the integral appearing in~\eqref{eq:main} is preserved by the push-forward
by the blow-up morphism $\pi: \Til V \to V$:
\begin{equation}\label{eq:pf0}
\pi_*\left(\int_{\hN} \frac{(n+1)! E\Til X_1\cdots \Til X_n da_0 da_1\cdots da_n}
{(1+a_0 E + a_1 \Til X_1+\cdots a_n \Til X_n)^{n+2}}\right)
= \int_N \frac{n! X_1\cdots X_n da_1\cdots da_n}{(1+a_1 X_1+\cdots a_n X_n)^{n+1}}
\quad,
\end{equation}
where $N$ is the Newton region for the intersection of monomials $D_1,\dots, D_r$
in the $X_i$'s, and~$\hN$ is the Newton region for the intersection of the
monomials $\pi^{-1}(D_1),\dots, \pi^{-1}(D_r)$ in $E$ and $\Til X_i$.
Here we are using coordinates $a_1,\dots$ corresponding to $X_1,\dots$,
and coordinates $a_0, a_1,\dots$ corresponding to $E, \Til X_1,\dots$.

Since the integral over the positive orthant is $1$, we may equivalently 
(see Remark~\ref{techn}) show that
\begin{equation}\label{eq:pf}
\pi_*\left(\int_{\hN'} \frac{(n+1)! E\Til X_1\cdots \Til X_n da_0 da_1\cdots da_n}
{(1+a_0 E + a_1 \Til X_1+\cdots a_n \Til X_n)^{n+2}}\right)
= \int_{N'} \frac{n! X_1\cdots X_n da_1\cdots da_n}{(1+a_1 X_1+\cdots a_n X_n)^{n+1}}
\quad,
\end{equation}
where $N'$, $\hN'$ are the complements of $N$, $\hN$ in the 
corresponding positive orthants. We will construct compatible triangulations $\hcU$
and $\cU$ of $\hN'$, $N'$ respectively, and use Lemma~\ref{simplexco} to analyze
the effect of $\pi_*$ on the corresponding contributions to the integrals in \eqref{eq:pf}.

\begin{remark}\label{techn}
A subtlety should be mentioned here.
On the face of it, \eqref{eq:pf} is an equality in the Chow group {\em of $V$,\/}
while our aim is to prove Theorem~\ref{main} as an equality {\em in~$A_*S$.\/} 
This is only an apparent difficulty. The integral $\int_{\hN'}$ on the left is shorthand
for its expansion as a power series in the parameters $\Til X_i$ (cf.~Corollary~\ref{pows}),
and this well-defined series has the form 
\[
1+\hp(E,\Til X_1,\dots, \Til X_n)
\]
where $\hp(E,\Til X_1,\dots, \Til X_n)$ is supported on $\pi^{-1}(S)$ after evaluation as a sum 
of intersection products in $\Til V$. Indeed, $\hp(E,\Til X_1,\dots, \Til X_n)$ is simply the $\int_{\hN}$
appearing in~\eqref{eq:pf0}. Likewise, the right-hand side of~\eqref{eq:pf} is a well-defined series
$1+p(X_1,\dots, X_n)$, where $p(X_1,\dots, X_n)$ is a sum of terms supported on $S$,
equaling the $\int_{N}$ on the right-hand side of~\eqref{eq:pf0}.
The sense in which \eqref{eq:pf} should be interpreted, and in which it will be proven, is that
the push-forward {\em of each\/} term in the series on the left contributes to a summand in
the series on the right. With the exception of $\pi_*(1)=1$, all these these push-forwards 
map classes in $A_*\pi^{-1}(S)$ to classes in $A_*S$. In particular, 
$\pi_*( \hp(E,\Til X_1,\dots, \Til X_n)) = p(X_1,\dots, X_n)$ {\em in $A_*S$,\/} which is 
precisely~\eqref{eq:pf0}.
\qede\end{remark}

With the notation introduced in \S\ref{intra}, we view $N'$ as the convex hull of 
\[
\uv_1,\cdots, \uv_r\quad;\quad \ua_1,\dots, \ua_n
\]
where $\uv_i$, $i=1,\dots,r$, are the lattice points corresponding to the 
monomials $D_i$. Likewise, $\hN'$ is the convex hull of
\[
\huv_1,\cdots, \huv_r\quad; \quad \ua_0,\ua_1,\dots, \ua_n
\]
where $\huv_i$, $i=1,\dots,r$, correspond to $\pi^{-1}(D_i)$. By Remark~\ref{nota},
the points $\huv_i$ are the lifts of the points $\uv_i$ to the hyperplane $H$ in 
$\Rbb^{n+1}$ with equation $a_0=a_1+a_2$. We let $M=\{\uv_1,\cdots, \uv_r\}$,
$\hM=\{\huv_1,\cdots, \huv_r\}$.
Note that $\ua_3, \dots, \ua_n$ belong to $H$, while $\ua_0, \ua_1, \ua_2$ do not. 
In fact, $\ua_0$ belongs to one of the two half-spaces determined by $H$ and
$\ua_1, \ua_2$ to the other.

To obtain the triangulations $\hcU$, $\cU$, we use the following procedure.

\begin{itemize}
\item Let $\cT$ be any triangulation of the convex hull of $M\cup \{\ua_3,\dots, \ua_n\}$.
\item Let $\hcT$ be the lift of $\cT$ to the hyperplane $H$. 
This is a triangulation of the convex hull of $\hM\cup \{\ua_3,\dots, \ua_n\}$.
\item Let $\hcU_1$ be the triangulation of the convex hull of $\hM\cup \{\ua_1, 
\ua_2,\underline a_3,\dots, \underline a_n\}$ obtained by first taking {\em pyramids\/} 
over all simplices in $\hcT$ with apex $\ua_1$ (cf.~\cite{MR2743368}, 
\S4.2.1), then {\em placing\/} $\ua_2$ on the resulting triangulation (cf.~\cite{MR2743368}, \S4.3.1).
\item Complete $\hcU_1$ to a triangulation $\hcU$ of $\hN'$, by placing $\ua_0$.
\end{itemize}

\begin{lemma}\label{trialem}
With notation as above:

(i) Each simplex $\sigma$ of top dimension ($=n$) in $\hcT$ determines {\em two\/}
simplices $\hsigma_0$, resp., $\hsigma_1$ of top dimension ($=n+1$) in $\hcU$, 
namely the pyramids over $\sigma$ with apex $\ua_0$, resp., $\ua_1$.

(ii) Every top-dimensional simplex in $\hcU$ including $\ua_2$ also includes $\ua_0$
or $\ua_1$.
\end{lemma}

\begin{proof}
Both points are direct consequences of the construction.

For (i), observe that all simplices of $\hcT$ (hence contained in $H$) are visible 
from both $\ua_0$ and $\ua_1$, since these points are on opposite sides of $H$, and 
$\ua_2$ is placed after $\ua_1$ in the construction.

For (ii), note that the top dimensional simplices in $\hcT$ are not visible from $\ua_2$ 
in the construction, since $\ua_2$ is on the same side of $H$ as $\ua_1$, and it is placed 
after $\ua_1$.
\end{proof}

\begin{itemize}
\item By construction, $N'$ is the convex hull of the projection in the $a_0$ direction 
of $\hM$ and $\ua_1,\dots, \ua_n$; that is, it is the {\em contraction\/} of $\hN'$
in the sense of~\cite{MR2743368}, Definition~4.2.19. By Lemma~4.2.20 
in~\cite{MR2743368}, we obtain a triangulation $\cU$ of $N'$ by taking the {\em links\/}
of $\ua_0$ with respect to $\hcU$ (cf.~Definition~2.1.6 in~\cite{MR2743368}).
\end{itemize}

The simplices in $\cU$ correspond precisely to the simplices in $\hcU$
of which $\ua_0$ is a vertex. 

We have constructed related triangulations $\hcU$ of $\hN'$ and $\cU$ 
of $N'$, and we have to study the effect of $\pi_*$ on the contributions to $\int_{\hN'}$
due to the top-dimensional simplices in $\hcU$ according to Lemma~\ref{simplexco}. 
We let $U$, resp., $\hU$ be the set of top-dimensional simplices in $\cU$, resp., $\hcU$.
The simplices in $U$ are in one-to-one correspondence with the simplices in $\hU$
containing $\ua_0$.

Each top dimensional simplex $\hsigma$ in $\hcU$ is the convex hull of an 
$r$-dimensional face in $\hcT$ and $n+2-r$ points at infinity, which may or 
may not include $\ua_0$, $\ua_1$, $\ua_2$. We split the set~$\hU$ into a disjoint
union $\hU_0 \amalg \hU_1 \amalg \hU' \amalg \hU''$, according to these different
possibilities.
\begin{itemize}
\item $\hU_0$ consists of the simplices $\hsigma$ which include $\ua_0$
and none of $\ua_1$, $\ua_2$. 
\item $\hU_1$ likewise consists of the simplices $\hsigma$ which include 
$\ua_1$ and none of $\ua_0$, $\ua_2$. 
\item $\hU'$ consists of the simplices $\hsigma$ which contain $\ua_0$ and at 
least one of $\ua_1$, $\ua_2$.
\item $\hU''$ consists of the simplices $\hsigma$ which do not include $\ua_0$,
and either include both or none of $\ua_1$ and $\ua_2$.
\end{itemize}
The remaining possibility, i.e., simplices which contain $\ua_2$ and none of
$\ua_0$, $\ua_1$, is excluded by Lemma~\ref{trialem} (ii).
The simplices in $\hU_0$, resp., $\hU_1$ are pyramids over top-dimensional
simplices in $\hcT$ with apex $\ua_0$, resp., $\ua_1$ (Lemma~\ref{trialem} (i)).

As noted above, by construction there is a one-to-one correspondence between 
$\hU'\cup \hU_0$ and $U$, associating with each $\hsigma\in \hU'\cup \hU_0$ 
the link of $\ua_0$ with respect to $\hsigma$. 
This natural bijection $\hU'\cup \hU_0 \to U$ is {\em not\/} 
\mymarginpar{\dbend}
compatible with push-forward at the level of the contributions to the integral $\int_{N'}$.
However, the decomposition found above allows us to define {\em another\/} 
realization of $U$. Define $\alpha: \hU' \amalg \hU_1 \to U$ as follows:
\begin{itemize}
\item If $\hsigma\in \hU'$, let $\alpha(\hsigma)$ be the link of $\ua_0$ with respect
to $\hsigma$.
\item If $\hsigma\in \hU_1$, then $\hsigma$ is the pyramid over an $n$-dimensional
simplex $\sigma$ in $\hcT$ with apex~$\ua_1$. Let $\alpha(\hsigma)$ be the link of 
$\ua_0$ with respect to the pyramid $\hsigma$ over $\sigma$ with apex $\ua_0$.
\end{itemize}
In other words, in the second case $\alpha(\hsigma)$ is the simplex in $\cT$
(and hence in $\cU$) corresponding to the simplex $\sigma$ of $\hcT$.
The function $\alpha$ is evidently a bijection.

\begin{lemma}\label{pflemma}
Let $\hsigma\in \hU$.
\begin{itemize}
\item If $\hsigma \in \hU_0 \amalg \hU''$, then the contribution of $\hsigma$
to the integral over $\hN'$ pushes forward to~$0$.
\item If $\hsigma \in \hU' \amalg \hU_1$, then the contribution of $\hsigma$
to the integral over $\hN'$ pushes forward to the contribution of $\alpha(\hsigma)\in U$
to the integral over $N'$.
\end{itemize}
\end{lemma}

Since $\hU=\hU_0 \amalg \hU_1 \amalg \hU' \amalg \hU''$ and $\alpha$ is a bijection
onto $U$, this lemma verifies \eqref{eq:pf}, proving Lemma~\ref{industep}
and hence concluding the proof of Theorem~\ref{main}.

By construction, each simplex $\hsigma$ in $\hU$ has a set of finite vertices 
$\huv_0, \dots, \huv_r$ {\em in the hyperplane~$H$\/}, and a set of infinite vertices.
According to Lemma~\ref{simplexco}, the corresponding contribution equals
\[
\frac{\hVol(\hsigma)\, C} {\prod_{\ell = 0}^r (1+\huv_\ell \cdot \Til \uX)}
\]
where $\huv\cdot \Til \uX=v_0 E+ v_1 \Til X_1 + \cdots$, and $C$ is
a product of divisors from $E, \Til X_1, \dots$. A key observation here is that
{\em if $\huv=(v_0,\dots, v_n)$ is the lift to $H$ of a corresponding vertex 
$\uv=(v_1,\dots, v_n)$,\/} then $v_0=v_1+v_2$, and it follows that
\[
(1+\huv \cdot \Til \uX) = \pi^* (1 + \uv \cdot \uX)\quad.
\]
Thus, the `denominator' in the contribution of $\hsigma$ is a pull-back.
Also, if $\sigma$ is the contraction of a simplex $\hsigma$ with respect to $a_0$,
then the finite vertices of $\sigma$ are precisely the projections $\uv_\ell$, and 
$\hVol(\sigma)=\hVol(\hsigma)$. By the projection formula, we see that
\[
\pi_*\left(
\frac{\hVol(\hsigma)\, C} {\prod_{\ell = 0}^r (1+\huv_\ell \cdot \Til \uX)}
\right)
=
\frac{\hVol(\sigma)\, \pi_*(C)} {\prod_{\ell = 0}^r (1+\uv_\ell \cdot \uX)}
\quad.
\]
This is the contribution of $\sigma$ to $\int_{N'}$, provided that $\pi_*(C)$ equals
the correct product of divisors corresponding to the infinite vertices of $\sigma$.
These are the infinite vertices of $\hsigma$, with $\ua_0$ removed. In the proof
of Lemma~\ref{pflemma}, $\pi_*(C)$ is either $0$ (in the first point listed in the
lemma) or equals the correct product (in the second).

\begin{proof}
Assume first $\hsigma\in \hU_0\amalg \hU''$. 

If $\hsigma\in \hU_0$, then $\hsigma$ includes $\ua_0$ and 
neither $\ua_1$ nor $\ua_2$. According to Lemma~\ref{simplexco} and the
discussion preceding this proof, the contribution 
of $\hsigma$ to $\int_{\hN'}$ has the form
\[
\frac{\hVol(\hsigma) \Til X_1\cdot \Til X_2 \cdot X_{i_1}\cdots}{\pi^*(\cdots)}
\]
with all the $i_j\ge 3$. (We are now using our convention of writing $X_i$ for $\Til X_i$
for $i\ge 3$, cf.~Remark~\ref{nota}.) This term equals $0$, since 
$\Til X_1\cap \Til X_2=\emptyset$.
If $\hsigma\in \hU''$, then $\hsigma$ does not contains~$\ua_0$ and contains
either both or neither of $\ua_1$ and $\ua_2$. Its contribution has the form
\[
\frac{\hVol(\hsigma) \, E\cdot X_{i_1}\cdots}{\pi^*(\cdots)}
\quad\text{or}\quad
\frac{\hVol(\hsigma)\, E\cdot \Til X_1\cdot \Til X_2 \cdot X_{i_1}\cdots}{\pi^*(\cdots)}
\]
with $i_j\ge 3$. In the first case it pushes forward to $0$ by the projection formula, and in 
the second case it is $0$, again because $\Til X_1\cap \Til X_2=\emptyset$.

This concludes the proof of the first part of the claim.

For the second part, assume $\hsigma\in \hU'\amalg \hU_1$. 
if $\hsigma \in \hU'$, then $\hsigma$ contains $\ua_0$ and $\ua_1$ or $\ua_2$ or both. 
The contribution of $\hsigma$ to $\int_{\hN'}$ has the form
\[
\frac{\hVol(\hsigma) \, X_{i_1}\cdots}{\prod_{\ell = 0}^r (1+\huv_\ell \cdot \Til \uX)}
\quad\text{or}\quad
\frac{\hVol(\hsigma) \, \Til X_\ell\cdot X_{i_1}\cdots}
{\prod_{\ell = 0}^r (1+\huv_\ell \cdot \Til \uX)}
\]
with all the $i_j\ge 3$ and $\ell =1$ or $2$. By the projection formula, these terms push 
forward to the corresponding contributions of $\alpha(\hsigma)$ to the integral over $N'$.

If $\hsigma\in \hU_1$, then $\hsigma=\hsigma_1$ for a top-dimensional
simplex $\sigma=\alpha(\hsigma)$ in $\hcT$; $\hsigma$ includes 
$\ua_1$, and does not include $\ua_0$ and $\ua_2$. 
The contribution of $\hsigma$ to $\int_{\hN'}$ has the form
\[
\frac{\hVol(\hsigma)\, E\cdot \Til X_2\cdot X_{i_1}\cdots}
{\prod_{\ell = 0}^r (1+\huv_\ell \cdot \Til \uX)}
\]
By the projection formula (cf.~Remark~\ref{nota}), this term pushes forward to a 
contribution
\[
\frac{\hVol(\hsigma)\, X_1\cdot X_2\cdot X_{i_1}\cdots}{\prod_{\ell = 0}^r 
(1+\uv_\ell \cdot \uX)}\quad,
\]
matching the contribution of $\alpha(\sigma)$, and concluding the proof.
\end{proof}

\subsection{An example}
A concrete example may clarify the argument presented in the previous section.
Consider the ideal $(x^3, xy, y^3)$. 
The shaded area in the following picture depicts~$N'$ in the plane $\Rbb^2$ with
coordinates $(a_1, a_2)$:
\begin{center}
\includegraphics[scale=.5]{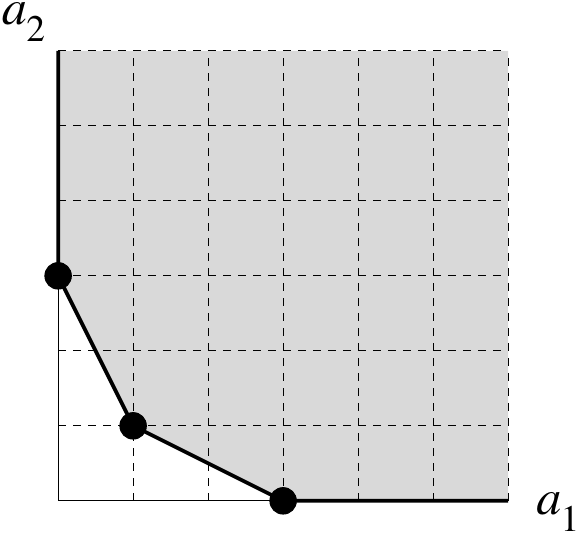}
\end{center}
Lifting to $\Rbb^3$, with `vertical' coordinate $a_0$:
\begin{center}
\includegraphics[scale=.6]{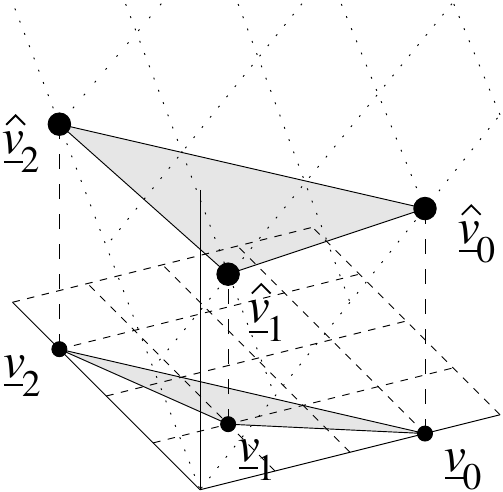}
\end{center}
Here we have shaded the triangle determined by the three monomials in the $(a_1,a_2)$
plane, as well as its lift to the hyperplane $H$ with equation $a_0=a_1+a_2$. 

With notation as in~\S\ref{prindu}, $\cT$ and $\hcT$ consist of the shaded triangles along
with their faces. The one-dimensional faces of $\hcU_1$ are included in the following
picture.
\begin{center}
\includegraphics[scale=.6]{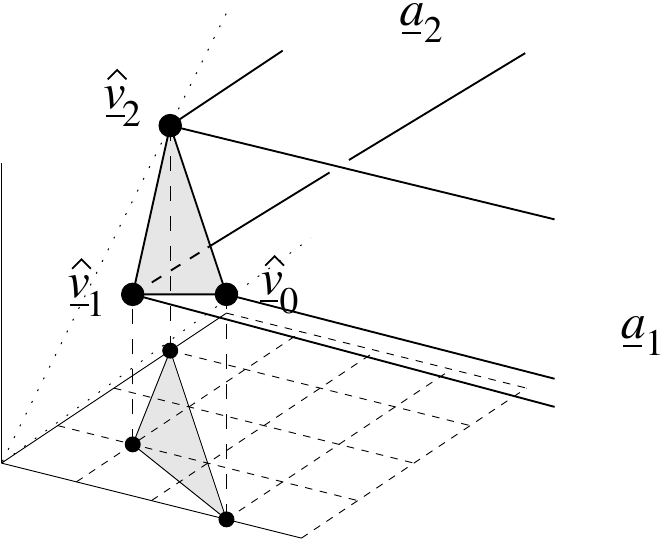}
\end{center}
The top-dimensional simplices in $\hcU_1$ are
\[
\huv_0\huv_1\huv_2\ua_1
\quad\text{and}\quad
\huv_1\huv_2\ua_1\ua_2\quad:
\]
the triangle $\huv_0\huv_1\huv_2$ is extended to a $3$-simplex in the $\ua_1$ direction;
the triangle $\huv_1\huv_2\ua_1$ is visible to $\ua_2$, so it produces a second $3$-simplex.
These two simplices and their faces form~$\hcU_1$. 
The vertex $\uv_0$ is not visible to $\ua_2$, since it is `behind' the plane containing $\huv_1$,
$\huv_2$, $\ua_1$. As remarked in Lemma~\ref{trialem} (ii),
a top-dimensional simplex including $\ua_2$ must also include~$\ua_1$.

The $2$-dimensional faces of $\hcU_1$ visible to $\ua_0$ are 
\[
\huv_0\huv_1\huv_2\quad,\quad
\huv_0\huv_2\ua_1\quad,\quad
\huv_2\ua_1\ua_2\quad.
\]
Therefore $\hcU$ consists of the five $3$-simplices
\[
\huv_0\huv_1\huv_2\ua_1\quad,\quad
\huv_1\huv_2\ua_1\ua_2\quad,\quad
\huv_0\huv_1\huv_2\ua_0\quad,\quad
\huv_0\huv_2\ua_1\ua_0\quad,\quad
\huv_2\ua_1\ua_2\ua_0
\]
and their faces. The corresponding triangulation $\cU$ is the projection of the faces
of $\hcU_1$ visible by $\ua_0$:
\begin{center}
\includegraphics[scale=.4]{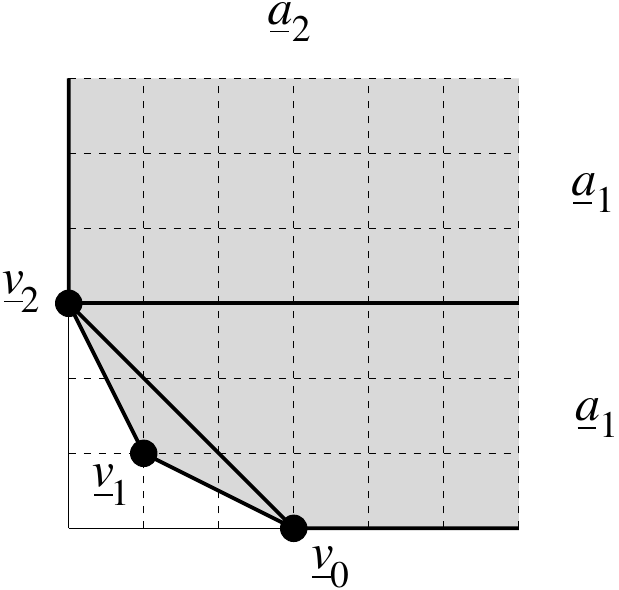}
\end{center}
The five simplices listed for $\hcU$ form the set $\hU$ used in the proof.
The decomposition $\hU_0 \amalg \hU_1\amalg \hU' \amalg \hU''$ is as
follows:
\[
\hU_0 = \{\huv_0\huv_1\huv_2\ua_0\},\quad
\hU_1 = \{\huv_0\huv_1\huv_2\ua_1\},\quad
\hU' = \{\huv_0\huv_2\ua_1\ua_0,\huv_2\ua_1\ua_2\ua_0\},\quad
\hU'' = \{\huv_1\huv_2\ua_1\ua_2\}\quad.
\]
The bijection $\alpha: \hU' \amalg \hU_1\to U$ maps 
\[
\begin{cases}
\huv_0\huv_2\ua_1\ua_0 \mapsto \uv_0\uv_2\ua_1\\
\huv_2\ua_1\ua_2\ua_0 \mapsto \uv_2\ua_1\ua_2 \\
\huv_0\huv_1\huv_2\ua_1 \mapsto \uv_0\uv_1\uv_2
\end{cases}\quad.
\]
Lemma~\ref{pflemma} now states that the push-forward $\pi_*$ will map
the contributions from
\[
\hU_0\amalg \hU'' = \{\huv_0\huv_1\huv_2\ua_0,\huv_1\huv_2\ua_1\ua_2\}
\]
to $0$, and those from 
\[
\hU' \amalg \hU_1 = \{\huv_0\huv_2\ua_1\ua_1\ua_0, \huv_2\ua_1\ua_1\ua_2\ua_0,
\huv_0\huv_1\huv_2\ua_1\ua_1\}
\]
to the total contributions of the simplices in $U$.
The contributions from the first set are
\[
\frac{3\Til X_1\Til X_2}{(1+3\Til X_1+3E)(1+\Til X_1+\Til X_2+2E)(1+3\Til X_2+3E)}
+\frac{E}{(1+\Til X_1+\Til X_2+2E)(1+3\Til X_2+3E)}
\]
and vanish in the push-forward as prescribed by Lemma~\ref{pflemma}. 
Those from the second,
\begin{multline*}
\frac{3 \Til X_2}{(1+3\Til X_1+3E)(1+3\Til X_2+3E)}
+\frac{1}{(1+3\Til X_2+3E)}\\
+\frac{3\Til X_2 E}{(1+3\Til X_1+3E)(1+\Til X_1+\Til X_2+2E)(1+3\Til X_2+3E)}
\end{multline*}
push-forward to
\[
\frac{3 X_2}{(1+3X_1)(1+3X_2)}
+\frac{1}{(1+3X_2)}
+\frac{3X_1 X_2}{(1+3X_1)(1+X_1+X_2)(1+3X_2)}
\]
that is, to the sum of contributions corresponding to the triangulation $\cU$ of $N'$.
\qede

%%%

\end{document}